\documentclass[a4paper,12pt]{article}
\usepackage{amsfonts}
\usepackage{amssymb,amsmath}
\usepackage{amsthm}
\usepackage{epsfig}                  
\usepackage{epstopdf}                
\usepackage{graphicx}
\usepackage{multirow}
\usepackage{mathrsfs}
\usepackage{comment}
\usepackage{bm}
\usepackage{booktabs}
\usepackage{enumitem}
\usepackage{listings}
\usepackage{color}
\usepackage[title]{appendix}
\usepackage{float}
\usepackage{algorithm}
\usepackage{algpseudocode}
\usepackage{fancyhdr}
\usepackage{setspace}
\usepackage{cases} 
\usepackage{chngcntr}
\usepackage{graphicx}
\usepackage{subcaption}
\usepackage{rotating}

\counterwithin{algorithm}{section} 

\usepackage{caption}
\DeclareMathOperator{\Fix}{Fix}
\DeclareMathOperator*{\argmin}{arg\,min}

\newtheorem{theorem}{Theorem}[section]

\newtheorem{lemma}{Lemma}[section]

\newtheorem{remark}{Remark}[section]

\newtheorem{example}{Example}[section]

\begin{document}

\title{\textbf{Convergence rate of the Halpern iterations
with possibly distinct anchor and initial guess}}

\author{
{\sc Jianing He$^{1}${\thanks{Corresponding author. email: hejianing@tjcu.edu.cn }}}\,\, and \,\,
{\sc Qiao-Li Dong$^{2}${\thanks{Corresponding author. email: dongql@lsec.cc.ac.cn}}}\\
{\footnotesize $^1$School of Information Engineering, Tianjin University of Commerce, Tianjin 300134, China}\\
{\footnotesize $^2$College of Science, Civil Aviation University of China, Tianjin 300300, China}\\
}
\date{}

\maketitle

\begin{abstract}
All existing convergence rate estimates of Halpern iterations are established only for   the case where the anchor coincides with the initial guess. This paper aims to investigate the convergence rate of  general Halpern iterations,  where the anchor and the initial guess may not necessarily be the same.  We present tight convergence rate estimates for both predetermined and adaptive anchoring parameters.
   These results generalize   existing related work.
\end{abstract}



\section{Introduction }
The calculation of fixed points for nonlinear operators, especially nonexpansive mappings, has always been an interesting topic due to its theoretical significance and practical value. The Halpern iteration is one of the most commonly used algorithms for finding fixed points of nonexpansive mappings.
\vskip 1mm
Let $\mathcal{H}$ be a real Hilbert space with  inner product $\langle\cdot,\cdot\rangle$ and norm $\|\cdot\|$ respectively and let $T: \mathcal{H}\rightarrow \mathcal{H}$ be a nonexpansive mapping. We use $\Fix(T)$ to denote the set of fixed points of $T$, that is, $\Fix(T)=\{x\in \mathcal{H}\, |\, x=Tx\}$.
The Halpern iteration   was first proposed by Halpern \cite{Halpern1967} in 1967 in  a Hilbert space. It generates,
 with an initial guess $x^0\in \mathcal{H}$ arbitrarily chosen,  a sequence $\{x^k\}_{k=0}^\infty$ by the iteration process:
 \begin{equation}\label{Halpern-1}
 x^k=\lambda_k u+(1-\lambda_k)Tx^{k-1},\qquad k=1,2,\ldots,
 \end{equation}
where $u$ is a fixed point in $\mathcal{H}$, referred to as anchor, and the parameters $\{\lambda_k\}_{k=1}^\infty$ are in $(0,1)$,
which will be referred to as anchoring parameters. Some researchers developed the sufficient  conditions on anchoring parameters for guaranteeing the strong convergence of Halpern iteration \cite{Halpern1967,Lio77,Wittmann1992,Rei94,Xu2002}.  He et al. \cite{He2019} first studied the adaptive selections of optimal anchoring parameters for the Halpern iteration.

Estimating the convergence rate of an algorithm is a very important topic. To estimate the convergence rate of the Halpern iteration,  several researchers have studied the properties of the displacement $\|x^k-Tx^k\|$, see \cite{He2019} and \cite{LL2007,KU2011,KU2012}.
 For more details, the reader is referred to the survey article \cite{Lop10}.

\vskip 1mm

Recently, there have been two important progresses in estimating the convergence rate of the Halpern iteration. The first one is about the Halpern iteration with $u=x^0$ and $\lambda_k=\frac{1}{k+1}$.  Lieder \cite{Lieder2021} presented
 the tight convergence rate estimate
\begin{equation}\label{rate-L}
\|x^k-Tx^k\|\leq \frac{2}{k+1}\|x^0-x^*\|,\quad k\geq 1,
\end{equation}
where  $x^*$ is an arbitrary fixed point of $T$.
The second one is about the  Halpern iteration:
\begin{equation}\label{Algorithm-He et al.}
x^k=\frac{1}{\varphi_k+1}x^0+\frac{\varphi_k}{\varphi_k+1}Tx^{k-1},
\end{equation}
with $u=x^0$ and
  \begin{equation}\label{eq3.1}
 \varphi_k:=\frac{2\langle x^{k-1}-Tx^{k-1}, x^0-x^{k-1}\rangle}{\|x^{k-1}-Tx^{k-1}\|^2}+1.
 \end{equation}
He et al.  \cite{He-Xu-Dong} provided the tight convergence rate estimate
\begin{equation}\label{rate}
\|x^k-Tx^k\|\leq \frac{2}{\varphi_k+1}\|x^0-x^*\|,\quad k\geq 1,
\end{equation}
 where  $x^*$ is an arbitrary fixed point of $T$. In general, (\ref{rate}) is  an improvement of
 (\ref{rate-L}) owing to $\varphi_k\geq k$ for all $k\geq 1$ (see \cite[Lemma 3.1(i)]{He-Xu-Dong}).
 These convergence rate estimates motivate the research on applying Halpern iteration to  accelerate iterative  methods in machine learning (generative adversarial networks (GANs), in particular) \cite{Dia20,Yoo21}
and structural optimizations (see, e.g., \cite{Lv-Dong2026,QiX21,Sun2025,ZZW2025}).
\vskip 1mm

We note that  the existing convergence rate estimates of Halpern iterations are all restricted to the case where the anchor $u$ and initial guess $x^0$ are the same. In this paper, we aim to investigate the convergence rate of  general Halpern iterations,  where $u$ and $x^0$ may be distinct.
We provide the tight convergence rate estimates for both predetermined and adaptive anchoring parameters, which  generalize  (\ref{rate-L}) and (\ref{rate}), respectively.

\vskip 1mm
The organization of the paper is as follows. In the next section we list several concepts and lemmas in a real Hilbert space that will be used in the convergence analysis of the algorithms. The main achievements of this paper are presented in Section 3.
 A brief summary and commentary on the research results of this paper are  provided in Section 4.

\vskip 2mm

\section{Preliminaries}

Let $\mathcal{H}$ be a real Hilbert space with inner product $\langle \cdot,\cdot\rangle$ and norm $\|\cdot\|$. Some concepts and tools are listed in this section, which will be used in the proofs of our main results. The following notations will be used
throughout the rest of the paper.
\begin{itemize}
  \item [(i)] $I: \mathcal{H}\rightarrow \mathcal{H}$ denotes the identity operator.
  \item [(ii)]  $\rightarrow $ and $\rightharpoonup$ denote the strong convergence and  the weak convergence, respectively.
  \item [(iii)] $\omega_w(x^k) :=\{x\mid\exists \,\, {\rm subsequence}\,\,   \{x^{k_i}\}_{i=0}^\infty\subset\{x^k\}_{k=0}^\infty$
  such that $ x^{k_i} \rightharpoonup x\}$ denotes the $\omega$-weak limit point set of $ \{x^k\}_{k=0}^\infty$.
 \end{itemize}

Let $C$ be a closed convex subset of $\mathcal{H}$. For an arbitrary point $u\in \mathcal{H}$,
the classical metric projection of
$u$ onto $C$, denoted by $P_Cu$, is defined by
$$
P_Cu:= \argmin\{  \| u- v\|\, |\,\,\forall v\in C\}.
$$

\begin{lemma} \label{le:2.2} {\rm(\cite[Section 3]{GR84})}
Let $C$ be a closed convex subset of $\mathcal{H}$. For $x\in \mathcal{H}$ and  $z\in C$, $z=P_{C}x$ if and only if
$$
\langle x-z,y-z \rangle\leq0,\quad \forall\,y\in C.
$$
\end{lemma}
Recall that a mapping $T: \mathcal{H}\rightarrow \mathcal{H}$ is called nonexpansive if
$$\|Tx-Ty\|\leq \|x-y\|$$
for all $x,y\in \mathcal{H}$.



\begin{lemma} {\rm (The demiclosedness principle for nonexpansive mappings \cite{GK1990})\label{lemmaGK}}
Let $C$ be a nonempty closed convex subset of a real Hilbert space $\mathcal{H}$ and let $T: C\rightarrow C$ be a nonexpansive mapping such that $\Fix(T)\neq\emptyset$. If a sequence $\{x^k\}_{k=0}^\infty$ in $C$ is such that $x^k\rightharpoonup z$ and $\|x^k-Tx^k\|\rightarrow 0$, then $z=Tz$.
\end{lemma}



The following result is elementary and can be verified straightforwardly. We formalize it as a lemma as it will be useful in our analysis.
\begin{lemma}
The identity
\begin{equation}\label{identity}
\varphi\|a\|^2+2\langle a,b\rangle+\|c\|^2-\|a+c\|^2=\frac{\varphi+1}{2}\|a\|^2-\frac{2}{\varphi+1}\|a+c-b\|^2+\frac{2}{\varphi+1}\|a+c-b-\frac{\varphi+1}{2}a\|^2
\end{equation}
holds for any $a,b,c\in \mathcal{H}$ and any number $\varphi>-1$.
\end{lemma}
\begin{lemma}{\rm\cite{Xu2002}}
\label{lem25}
Assume that $\{a_{k}\}_{k=0}^\infty$ is a sequence of nonnegative real numbers such that
$$ a_{k+1}\leq(1-\gamma_{k})a_{k}+\gamma_{k}\delta_{k},~k \geq 0,$$
where $\{\gamma_{k}\}_{k=0}^\infty$ is a sequence in $(0,1)$ and $\{\delta_{k}\}_{k=0}^\infty$ is a real sequence such that
\begin{itemize}
\item[{\rm(i)}]$\sum_{k=0}^{\infty} \gamma_k=\infty$;
\item[{\rm(ii)}]$ \limsup_{k\rightarrow \infty}\delta_{k}\leq 0~ or~ \sum_{k=0}^{\infty} |\gamma_k\delta_{k}|<\infty.$
\end{itemize}
Then $\lim_{k\rightarrow \infty} a_{k}=0.$
\end{lemma}

\section{Main results}
\vskip 1mm
Both estimates (\ref{rate-L}) by Lieder  and  (\ref{rate}) by He et al. are  established for nonexpansive mappings under the assumption that the anchor $u$ coincides with the initial guess $x^0$, i.e., $u=x^0$. To the best of our knowledge, no convergence rate estimates have been given for the general case where $u$ and $x^0$ are possibly distinct. In this section, we address this issue.
\vskip 2mm

\subsection{Convergence rates of general Halpern iterations with predetermined anchoring parameters}

In this subsection, we estimate the convergence rate of the following Halpern iteration
\begin{equation}\label{algorithm-(u,x0)}
x^k=\frac{1}{k+1}u+\frac{k}{k+1}Tx^{k-1},\quad k\geq 1,
\end{equation}
where the anchor $u$  and the initial guess $x^0$ are  chosen  arbitrarily in $\mathcal{H}$.

\begin{theorem}\label{theorem-4.1}
Assume that $T:\mathcal{H}\rightarrow \mathcal{H}$ is a nonexpansive mapping such that $\Fix(T)\neq \emptyset$.
Let $ \{x^k\}_{k=0}^\infty$ be a sequence generated by  the scheme \eqref{algorithm-(u,x0)}.
Then we have
\begin{equation}\label{rate-4-1}
\|x^k-Tx^k\|\leq \frac{2}{k+1}\left(\|u-x^0\|+\|u-x^*\|\right),\quad k\geq 1,
\end{equation}
where $x^*$ is an arbitrary fixed point of $T$.
The convergence rate estimate \eqref{rate-4-1} is tight.
\end{theorem}
\begin{proof}
Let $y^0=u$ and let $ \{y^k\}_{k=0}^\infty$ be a sequence generated by the scheme:
\begin{equation}\label{algorithm-(u=y0)}
y^k=\frac{1}{k+1}u+\frac{k}{k+1}Ty^{k-1},\quad k\geq 1.
\end{equation}
According to (\ref{rate-L}), one has that for any $x^*\in \Fix(T)$,
\begin{equation}\label{estimation-1}
\|y^k-Ty^k\|\leq\frac{2}{k+1}\|y^0-x^*\|,\quad \forall k\geq 1.
\end{equation}
On the other hand,  we get from (\ref{algorithm-(u,x0)}) and (\ref{algorithm-(u=y0)}) that
\begin{equation*}
\|x^k-y^k\|= \frac{k}{k+1}\|Tx^{k-1}-Ty^{k-1}\|\leq\frac{k}{k+1}\|x^{k-1}-y^{k-1}\|,\,\,\,\,\forall k\geq 1.
\end{equation*}
Consequently, it follows by induction that
\begin{equation}\label{estimation-2}
\|x^k-y^k\|\leq\frac{1}{k+1}\|x^{0}-y^{0}\|,\,\,\,\,\forall k\geq 1.
\end{equation}
Noting $y^0=u$, we obtain by combining (\ref{estimation-1}) and (\ref{estimation-2}) that
\begin{equation}
\aligned
\|x^k-Tx^k\|&\leq \|x^k-y^k\|+\|y^k-Ty^k\|+\|Ty^k-Tx^k\|\\
&\leq 2\|x^k-y^k\|+\|y^k-Ty^k\|\\
&\leq \frac{2}{k+1}\left(\|y^0-x^0\|+\|y^0-x^*\|\right)\\
&=\frac{2}{k+1}\left(\|u-x^0\|+\|u-x^*\|\right).
\endaligned
\end{equation}
Thus, (\ref{rate-4-1}) is established.
\end{proof}

The tightness of (\ref{rate-4-1}) will be confirmed by the following example.

\begin{example}
\label{example31}
   \rm
Let $\mathcal{H}=\mathbb{R}^2$ and  $T: \mathbb{R}^2\rightarrow \mathbb{R}^2$ be defined by
\begin{equation}\label{T}
T\left(
\aligned
&\xi\\
&\eta
\endaligned
\right)
=
\left(
\begin{matrix}
-1 & 0\\
0 & -1\\
\end{matrix}
\right)
\left(
\aligned
&\xi\\
&\eta
\endaligned
\right),\,\,\forall
\left(
\aligned
&\xi\\
&\eta
\endaligned
\right)
\in \mathbb{R}^2.
\end{equation}
Since $T$ is a counterclockwise $\pi$-rotation centered at $(0,0)^\top$,  it is obviously a  nonexpansive mapping with the unique fixed point $x^*=(0,0)^\top$.
\end{example}

Let $\{x^k\}_{k=0}^\infty$ be the sequence generated by the Halpern iteration (\ref{algorithm-(u,x0)}) with $u=(1,0)^\top$ and $x^0=(2, 0)^\top$. It is easy to get from (\ref{algorithm-(u,x0)}) and (\ref{T}) that
\begin{equation*}\label{Halpern xk}
x^{k}=\frac{1}{k+1}\{u+Tu+T^2u+\cdots+T^{k-1}u+T^kx^0\},\,\,\,\,\forall k\geq 1.
\end{equation*}
Consequently, we get
\begin{equation}\label{4-xk-Txk*}
x^k-Tx^k=\frac{1}{k+1}(u-T^ku+T^kx^0-T^{k+1}x^0),\,\,\,\,\forall k\geq 1.
\end{equation}
Note that $T$ is norm-preserving as a rotation operator, therefore for all $k\geq 1$,
\begin{equation}\label{rotation rate}
\begin{aligned}
\|x^k-Tx^k\|\geq&\frac{1}{k+1}(\|T^k(x^0-Tx^0)\|-\|u-T^ku\|)\\
\geq&\frac{1}{k+1}(\|x^0-Tx^0\|-2\|u\|)\\
=&\frac{2(\|x^0\|-\|u\|)}{k+1}\\
=&\frac{2}{k+1}.
\end{aligned}
\end{equation}
Combining (\ref{rate-4-1}) and (\ref{rotation rate}),  we assert that
 $\|x^k-Tx^k\|$ and $\frac{1}{k+1}$ are  infinitesimals  of  the same order. Hence the Halpern iteration (\ref{algorithm-(u,x0)}) has sublinear convergence rate. Moreover, using $T^2=I$, we also have from (\ref{4-xk-Txk*}) that
 $$\|x^{2k}-Tx^{2k}\|=\frac{1}{2k+1}\|x^0-Tx^0\|=\frac{4}{2k+1}=\frac{2}{2k+1}(\|u-x^0\|+\|u-x^*\|),\,\,\forall k\geq 1.$$
 This implies that (\ref{rate-4-1}) is tight.

\begin{remark}
\rm
When $u=x^0$, (\ref{rate-4-1}) becomes (\ref{rate-L}), hence the latter is a special case of the former.
\end{remark}

Next we investigate the convergence rate estimate for the Halpern iteration
\begin{equation}\label{algorithm-(u=2x0-Tx0)}
x^k=\frac{1}{k+3}u+\frac{k+2}{k+3}Tx^{k-1},\quad k\geq 1,
\end{equation}
where the initial guess $x^0$ is  chosen arbitrarily in $\mathcal{H}$  and the anchor $u=2x^0-Tx^0$.
\begin{theorem}\label{theorem-4.2}
Assume that $T:\mathcal{H}\rightarrow \mathcal{H}$ is a nonexpansive mapping such that $\Fix(T)\neq \emptyset$.
Let   $ \{x^k\}_{k=0}^\infty$ be a sequence generated by  the scheme \eqref{algorithm-(u=2x0-Tx0)}.
Then we have
\begin{equation}\label{rate-2-4}
\|x^k-Tx^k\|\leq \frac{2}{k+3}\|u-x^*\|,\quad k\geq 1,
\end{equation}
where $x^*$ is an arbitrary fixed point of $T$.   The convergence rate estimate \eqref{rate-2-4} is tight.
\end{theorem}

\begin{proof}
We first show (\ref{rate-2-4}).  It follows from (\ref{algorithm-(u=2x0-Tx0)}) that
$$x^k-x^{k-1}=\frac{1}{k+3}(u-Tx^{k-1})+(Tx^{k-1}-x^{k-1}).$$
Consequently, we obtain
\begin{equation}\label{x-x}
\aligned
&\|x^k-x^{k-1}\|^2\\
=&\frac{1}{(k+3)^2}\|u-Tx^{k-1}\|^2+\|x^{k-1}-Tx^{k-1}\|^2+\frac{2}{k+3}\langle u-Tx^{k-1}, Tx^{k-1}-x^{k-1}\rangle\\
=&\frac{1}{(k+3)^2}\|u-Tx^{k-1}\|^2+\frac{k+1}{k+3}\|x^{k-1}-Tx^{k-1}\|^2+\frac{2}{k+3}\langle u-x^{k-1}, Tx^{k-1}-x^{k-1}\rangle.
\endaligned
\end{equation}
Using (\ref{algorithm-(u=2x0-Tx0)}) again, we also have
$$x^k-u=\frac{k+2}{k+3}(Tx^{k-1}-u),$$
and
$$Tx^{k-1}-Tx^k=\frac{1}{k+2}(x^k-u)+(x^k-Tx^k),$$
which imply
\begin{equation}\label{di-xiao}
\frac{1}{(k+2)^2}\|x^k-u\|^2=\frac{1}{(k+3)^2}\|u-Tx^{k-1}\|^2,
\end{equation}
and
\begin{equation}\label{Tx-Tx}
\|Tx^{k-1}-Tx^k\|^2=\frac{1}{(k+2)^2}\|x^k-u\|^2+\|x^k-Tx^k\|^2+\frac{2}{k+2}\langle x^k-u, x^k-Tx^k\rangle,
\end{equation}
respectively. Using the nonexpansiveness of $T$ and combining (\ref{x-x})-(\ref{Tx-Tx}), we deduce
\begin{equation}\label{ditui-inequality}
\aligned
&\|x^k-Tx^k\|^2+\frac{2}{k+2}\langle x^k-u, x^k-Tx^k\rangle\\
\leq&\|x^{k-1}-x^k\|^2-\frac{1}{(k+2)^2}\|x^k-u\|^2\\
\leq &\frac{k+1}{k+3}\|x^{k-1}-Tx^{k-1}\|^2+\frac{2}{k+3}\langle u-x^{k-1}, Tx^{k-1}-x^{k-1}\rangle.
\endaligned
\end{equation}
Letting $k=1$ in (\ref{ditui-inequality}) and substituting $u=2x^0-Tx^0$ into its right-hand side  yields
\begin{equation}\label{induction begin}
\|x^1-Tx^1\|^2+\frac{2}{3}\langle x^1-u, x^1-Tx^1\rangle\leq \frac{1}{2}\|x^0-Tx^0\|^2+\frac{1}{2}\langle x^0-Tx^0, Tx^0-x^0\rangle=0.
\end{equation}
Based on (\ref{ditui-inequality}) and (\ref{induction begin}),  it can be easily proven by induction that
\begin{equation}\label{jiben-inequality}
\|x^k-Tx^k\|^2+\frac{2}{k+2}\langle x^k-u, x^k-Tx^k\rangle\leq 0,\,\,\,\forall k\geq 1,
\end{equation}
which yields
\begin{equation}
\|x^k-Tx^k\|\leq\frac{2}{k+2}\| x^k-u\|,\,\,\,\forall k\geq 1.
\end{equation}
Let $a=x^k-Tx^k$, $b=x^k-u$, $c=Tx^k-x^*$ and $\varphi=k+2$. Then $a+c=x^k-x^*$ and $a+c-b=u-x^*$. Using (\ref{identity}), we have
\begin{equation}\label{4.15}
\aligned
&(k+2)\|x^k-Tx^k\|^2+2\langle x^k-Tx^k, x^k-u\rangle+\|Tx^k-x^*\|^2-\|x^k-x^*\|^2\\
=& \frac{k+3}{2}\|x^k-Tx^k\|^2-\frac{2}{k+3}\|u-x^*\|^2+\frac{2}{k+3}\|u-x^*-\frac{k+3}{2}(x^k-Tx^k)\|^2.
\endaligned
\end{equation}
Meanwhile, we get from (\ref{jiben-inequality}) that
\begin{equation}\label{4.16}
(k+2)\|x^k-Tx^k\|^2+2\langle x^k-u, x^k-Tx^k\rangle+\|Tx^k-x^*\|^2-\|x^k-x^*\|\leq 0,\,\,\,\forall k\geq 1.
\end{equation}
Comparing (\ref{4.15}) and (\ref{4.16}), it is easy to see that
$$\frac{k+3}{2}\|x^k-Tx^k\|^2-\frac{2}{k+3}\|u-x^*\|^2\leq 0,$$
and consequently  (\ref{rate-2-4}) holds true.

Next, we show that (\ref{rate-2-4}) is tight. To this end, we  adopt the scheme (\ref{algorithm-(u=2x0-Tx0)}) to solve Example \ref{example31}, where $x^0\neq (0,0)^\top$  is chosen arbitrarily  and $u=2x^0-Tx^0$.
It is easy to get from (\ref{algorithm-(u=2x0-Tx0)}) and (\ref{T}) that
\begin{equation*}\label{Halpern xk}
x^{k}=\frac{1}{k+3}\left(u+Tu+T^2u+\cdots+T^{k-1}u+3T^kx^0\right),\,\,\,\,\forall k\geq 1.
\end{equation*}
Consequently, we get
\begin{equation}\label{4-xk-Txk}
x^k-Tx^k=\frac{1}{k+3}\left[u-T^ku+3(T^kx^0-T^{k+1}x^0)\right],\,\,\,\,\forall k\geq 1.
\end{equation}
Using $T^2=I$ and \eqref{4-xk-Txk}, we have
$$
\|x^{2k}-Tx^{2k}\|=\frac{3}{2k+3}\|x^0-Tx^0\|=\frac{6}{2k+3}\|x^0\|=\frac{2}{2k+3}\|u-x^*\|.
$$
This implies that (\ref{rate-2-4}) is tight.
\end{proof}

\begin{remark}
\rm
It is obvious that the convergence rate  \eqref{rate-2-4} is better than  (\ref{rate-4-1}). Furthermore, let $ \{x^k\}_{k=0}^\infty$ be generated by the scheme (\ref{algorithm-(u,x0)}) with $u=2x^0-Tx^0$,  then by (\ref{rate-4-1}) we obtain the following convergence rate estimate:
\begin{equation}\label{rate-4-17}
\|x^k-Tx^k\|\leq \frac{2}{k+1}\{\|x^0-Tx^0\|+\|u-x^*\|\},\quad k\geq 1,
\end{equation}
which is also worse than (\ref{rate-2-4}).
\end{remark}

\subsection{Convergence rate of general Halpern iterations with adaptive anchoring parameters}

In this subsection, we  consider the general Halpern iteration with adaptive anchoring parameters:
\begin{equation}\label{eq3.2}
x^k=\frac{1}{\varphi_k+1}u+\frac{\varphi_k}{\varphi_k+1}Tx^{k-1},\quad k\geq 1,
\end{equation}
where    $\{\varphi_k\}_{k=1}^\infty$ is given by
  \begin{equation}\label{eq3.1}
 \varphi_k:=\frac{2\langle x^{k-1}-Tx^{k-1}, u-x^{k-1}\rangle}{\|x^{k-1}-Tx^{k-1}\|^2}+1,\quad k\geq 1,
 \end{equation}
  and the anchor $u$ and the initial guess $x^0$  are chosen in $\mathcal{H}$ such that $\varphi_0:=\frac{2\langle x^0-Tx^0, u-x^0\rangle}{\|x^0-Tx^0\|^2}>-1$.

 \begin{remark}
\rm
Without loss of generality, we always assume that $x^k\neq Tx^k$ for all $k\geq 0$ in the algorithm (\ref{eq3.2})-(\ref{eq3.1}).
Namely, it generates an infinite sequence of iterates $\{x^k\}_{k=0}^\infty$.
\end{remark}

The following lemma provides the fundamental conclusions of the algorithm (\ref{eq3.2})-(\ref{eq3.1}).

\begin{lemma}\label{varphi-k}
Let  $\{x^k\}_{k=0}^\infty$ be generated by the algorithm (\ref{eq3.2})-(\ref{eq3.1}). Then the following properties hold for all $k\geq 1$,
\begin{itemize}
\item[{\rm(i)}] $\varphi_k\geq \varphi_0+k$,
\item[{\rm(ii)}] $\|x^k-Tx^k\|^2\leq\dfrac{2}{\varphi_k}\langle x^k-Tx^k, u-x^k\rangle.$
\end{itemize}
\end{lemma}

\begin{proof} We prove the conclusions (i) and (ii) by induction.
Noting $\varphi_0:=\frac{2\langle x^0-Tx^0, u-x^0\rangle}{\|x^0-Tx^0\|^2}>-1$, we get from (\ref{eq3.1}) that $\varphi_1=\varphi_0+1>0$ and this means that (i) holds for $k=1$.
From (\ref{eq3.2}), we derive that
\begin{equation}\label{eq3.5}
Tx^{0}=\frac{\varphi_1+1}{\varphi_1}x^1-\frac{1}{\varphi_1}u=x^1+\frac{1}{\varphi_1}(x^1-u).
\end{equation}
By the nonexpansiveness of $T$ and (\ref{eq3.5}), we have
\begin{align}\label{eq3.6}
\|x^{0}-x^1\|^2&\geq \|Tx^{0}-Tx^1\|^2=\|(x^1-Tx^1)+\frac{1}{\varphi_1}(x^1-u)\|^2 \nonumber\\
&=\|x^1-Tx^1\|^2+\frac{2}{\varphi_1}\langle x^1-Tx^1, x^1-u \rangle+\frac{1}{\varphi_1^2}\|x^1-u\|^2.
\end{align}
Using (\ref{eq3.2}) again, we obtain
\begin{align}\label{eq3.7}
\|x^{0}-x^1\|^2
&=\|(x^{0}-Tx^{0})-\frac{1}{\varphi_1+1}(u-Tx^{0})\|^2 \nonumber\\
&=\|x^{0}-Tx^{0}\|^2-\frac{2}{\varphi_1+1}\langle x^{0}-Tx^{0}, u-Tx^{0} \rangle  +\frac{1}{(\varphi_1+1)^2}\|u-Tx^{0}\|^2\nonumber\\
&=\|x^{0}-Tx^{0}\|^2-\frac{2}{\varphi_1+1}\langle x^{0}-Tx^{0}, u-Tx^{0} \rangle  +\frac{1}{\varphi_1^2}\|x^1-u\|^2.
\end{align}
Combining \eqref{eq3.6} and \eqref{eq3.7} implies
\begin{equation}\label{eq3.9-1}
\aligned
&\|x^1-Tx^1\|^2+\frac{2}{\varphi_1}\langle x^1-Tx^1, x^1-u \rangle\\
\leq&\|x^{0}-Tx^{0}\|^2-\frac{2}{\varphi_1+1}\langle x^{0}-Tx^{0}, u-Tx^{0} \rangle\\
=&\frac{\varphi_1-1}{\varphi_1+1}\|x^{0}-Tx^{0}\|^2-\frac{2}{\varphi_1+1}\langle x^{0}-Tx^{0}, u-x^{0} \rangle\\
=&0,
\endaligned
\end{equation}
where the second equality comes from  (\ref{eq3.1}).
Consequently,
$$\|x^1-Tx^1\|^2\leq\frac{2}{\varphi_1}\langle x^1-Tx^1, u-x^1 \rangle.$$
So (ii) also holds for $k=1$.

Suppose that (i) and (ii) hold for $k-1$ ($k\geq 2$), that is,
\begin{equation}\label{eq3.3}
\varphi_{k-1}\geq\varphi_0+k-1,
\end{equation}
and
\begin{equation}\label{eq3.4}
\|x^{k-1}-Tx^{k-1}\|^2\leq\frac{2}{\varphi_{k-1}}\langle x^{k-1}-Tx^{k-1},u-x^{k-1}\rangle.
\end{equation}
Then $\varphi_{k-1}\geq \varphi_0+k-1>k-2\ge 0$. By (\ref{eq3.1}),  (\ref{eq3.3}) and (\ref{eq3.4}), we get
\begin{equation*}
\varphi_k=\frac{2\langle x^{k-1}-Tx^{k-1}, u-x^{k-1}\rangle}{\|x^{k-1}-Tx^{k-1}\|^2}+1
\geq \varphi_{k-1}+1\geq \varphi_0+k,
\end{equation*}
that is, (i) holds for $k$.
Similar to the derivation of \eqref{eq3.9-1}, we have
\begin{equation*}
\|x^k-Tx^k\|^2+\frac{2}{\varphi_k}\langle x^k-Tx^k, x^k-u \rangle
\le0,
\end{equation*}
which yields
$$\|x^k-Tx^k\|^2\leq\frac{2}{\varphi_k}\langle x^k-Tx^k, u-x^k \rangle.$$
Hence (ii) also holds for $k$ and  the proof is finished.
\end{proof}

\begin{theorem}\label{theorem-4.1}
Assume that $T:\mathcal{H}\rightarrow \mathcal{H}$ is a nonexpansive mapping such that $\Fix(T)\neq \emptyset$. Let   $ \{x^k\}_{k=0}^\infty$ be a sequence generated by the algorithm (\ref{eq3.2})-(\ref{eq3.1}).
Then we have
\begin{equation}\label{rate-adaptive}
\|x^k-Tx^k\|\leq \frac{2}{\varphi_k+1}\|u-x^*\|,\quad k\geq 1,
\end{equation}
 where $x^*$ is an arbitrary fixed point of $T$. The convergence rate estimate \eqref{rate-adaptive} is tight.
\end{theorem}

\begin{proof}
Set $\varphi:=\varphi_k$,
$a:=x^k-Tx^k$, $b:=x^k-u$, and $c:=Tx^k-x^*$, then $a+c=x^k-x^*$, $a+c-b=u-x^*$.  By  (\ref{identity}), it follows  that
\begin{align*}
&\varphi_k \|x^k-Tx^k\|^2+2\langle x^k-Tx^k, x^k-u\rangle+\|Tx^k-x^*\|^2-\|x^k-x^*\|^2 \nonumber\\
&=\frac{\varphi_k+1}{2}\|x^k-Tx^k\|^2-\frac{2}{\varphi_k+1}\|u-x^*\|^2
+\frac{2}{\varphi_k+1}\|u-x^*-\frac{\varphi_k+1}{2}(x^k-Tx^k)\|^2.
\end{align*}
Combining the above equality,   Lemma \ref{varphi-k}(ii) and the nonexpansiveness of $T$, we have
$$\frac{\varphi_k+1}{2}\|x^k-Tx^k\|^2-\frac{2}{\varphi_k+1}\|u-x^*\|^2\leq 0.$$
It is immediately clear that the estimate (\ref{rate-adaptive}) follows.

The estimate (\ref{rate-adaptive}) is tight, which can be shown by Example 3.1 in \cite{Lieder2021}. In fact, by direct calculation,
it is easy to verify that for all $k\geq 1$, $\varphi_k=k$ and (\ref{rate-adaptive}) becomes an equality for this example.
\end{proof}
\begin{remark}
\rm
(i) When $u=x^0$, (\ref{rate-adaptive}) becomes (\ref{rate}), thus the latter is a special case of the former.\\
(ii) By Lemma \ref{lemmaGK} and Lemma \ref{lem25}, following a similar argument as in the proof of  Theorem 3.1 in \cite{He-Xu-Dong}, it is easy to show that under the conditions of Theorem \ref{theorem-4.1}, the iterative sequence generated by the algorithm (\ref{eq3.2})-(\ref{eq3.1}) strongly converges to a fixed point of $T$. The proof is omitted here.
\end{remark}

\section{Some concluding remarks}

In this paper, we provide the tight convergence rate estimates  of general Halpern iterations for both predetermined and adaptive anchoring parameters for the case where the anchor may not coincide with the initial guess. These achievements generalize   existing related results. How to further improve the convergence rate of the Halpern iteration deserves further in-depth exploration.

\vskip 3mm

\noindent
{\bf Conflict of Interest:}
The authors declare that they have no conflict of interest.

\end{document}